\documentclass[reqno,12pt]{amsart}
\usepackage{amsfonts}
\usepackage{bbm}
\usepackage{} 
\numberwithin{equation}{section}
\usepackage{indentfirst}
 \usepackage{color}
\usepackage{amssymb}
\usepackage{mathrsfs}
\usepackage{xy}
\xyoption{all}
\def\Ext{\mbox{\rm Ext}\,} \def\Hom{\mbox{\rm Hom}} \def\dim{\mbox{\rm dim}\,} \def\Iso{\mbox{\rm Iso}\,}
\def\lr#1{\langle #1\rangle}    \def\Tor{\mbox{\rm Tor}\,}\def\mod{\mbox{\rm \textbf{mod}}\,}
\def\Ker{\mbox{\rm Ker}\,}   \def\im{\mbox{\rm Im}\,}  \def\ra{\rightarrow}
\def\End{\mbox{\rm End}\,}\def\C{\mathcal {C}}
 \def\gl.{\mbox{\rm gl.}\,}

\def\A{\mathcal{A}\,} 

\def\bbz{{\mathbb Z}}
\theoremstyle{plain} 
\newtheorem{theorem}{\bf Theorem}[section]
\newtheorem{lemma}[theorem]{\bf Lemma}
\newtheorem{corollary}[theorem]{\bf Corollary}
\newtheorem{proposition}[theorem]{\bf Proposition}

\theoremstyle{definition} 
\newtheorem{definition}[theorem]{\bf Definition}
\newtheorem{remark}[theorem]{\bf Remark}
\newtheorem{example}[theorem]{\bf Example}

\newcommand{\bt}{\begin{theorem}}
\newcommand{\et}{\end{theorem}}
\newcommand{\bl}{\begin{lemma}}
\newcommand{\el}{\end{lemma}}
\newcommand{\bd}{\begin{definition}}
\newcommand{\ed}{\end{definition}}
\newcommand{\bc}{\begin{corollary}}
\newcommand{\ec}{\end{corollary}}
\newcommand{\bp}{\begin{proof}}
\newcommand{\ep}{\end{proof}}
\newcommand{\bx}{\begin{example}}
\newcommand{\ex}{\end{example}}
\newcommand{\br}{\begin{remark}}
\newcommand{\er}{\end{remark}}
\newcommand{\be}{\begin{equation}}
\newcommand{\ee}{\end{equation}}
\newcommand{\ba}{\begin{align}}
\newcommand{\ea}{\end{align}}
\newcommand{\bn}{\begin{enumerate}}
\newcommand{\en}{\end{enumerate}}
\newcommand{\bcs}{\begin{cases}}
\newcommand{\ecs}{\end{cases}}
\makeatletter
\renewcommand{\section}{\@startsection{section}{1}{0mm}
  {-\baselineskip}{0.5\baselineskip}{\bf\leftline}}
\makeatother

\begin{document}

\title[Multiplicative properties in Bridgeland's Hall algebras]{Multiplicative properties of certain elements\\ in Bridgeland's Hall algebras} 

\author{Qinghua Chen and Haicheng Zhang$^{*}$} 



\thanks{$*$~Corresponding author. Supported by the National Natural Science Foundation of
China (Grant No.11701090) and the Natural Science Foundation of Fujian Province of China (Grant No.2017J05004).}

\address{
College of Mathematics and Computer Science, Fu Zhou University, Fu Zhou 350108, P. R. China.\endgraf
}
\email{chenqh@mail.bnu.edu.cn}

\address{
Institute of Mathematics, School of Mathematical Sciences, Nanjing Normal University,
 Nanjing 210023, P. R. China.\endgraf
}
\email{zhanghai14@mails.tsinghua.edu.cn}


\maketitle

\begin{abstract}
Let $A$ be a finite dimensional algebra of finite global dimension over a finite field. In the present paper, we introduce certain elements in Bridgeland's Hall algebra of $A$, and give a multiplication theorem of these elements. In particular, this generalizes the main result in \cite{GP}.
\end{abstract}

\section{Introduction}
The Hall algebra $\mathfrak{H}(A)$ of a finite dimensional algebra $A$ over a finite field was introduced by Ringel \cite{R90a} in 1990. Ringel \cite{R90,R90a} proved that if $A$ is representation-finite and hereditary, the twisted Hall algebra $\mathfrak{H}_{v}(A)$, called the Ringel--Hall algebra, is isomorphic to the positive part of the corresponding quantized enveloping algebra.  By introducing a bialgebra structure on $\mathfrak{H}_{v}(A)$,
Green \cite{Gr95} generalized Ringel's work to an arbitrary finite dimensional hereditary algebra $A$
and showed that the composition subalgebra of $\mathfrak{H}_{v}(A)$ generated by simple $A$-modules
gives a realization of the positive part of the quantized enveloping algebra associated with $A$. In \cite{Xiao}, Xiao gave a realization of the whole quantized enveloping algebra by constructing the Drinfeld double of the extended Ringel--Hall algebra of a hereditary algebra.

In order to give an intrinsic realization of the entire quantized enveloping algebra via Hall algebra approaches, one has managed to define the Hall algebra of a triangulated category satisfying some homological finiteness conditions (for example, \cite{Kapranov}, \cite{Toen}, \cite{XiaoXu}). Unfortunately, the root category of a finite dimensional algebra does not satisfy the homological finiteness conditions. In other word, the Hall algebra of a root category has no a befitting definition. In fact, more generally, the Hall algebra of an odd periodic triangulated category with some finiteness conditions has been defined (see \cite{XC}). Nevertheless, an applicable definition of the Hall algebra of an even periodic triangulated category has been unknown so far.

In 2013, for each finite dimensional algebra $A$ of finite global dimension, Bridgeland \cite{Br} introduced an algebra, called the (reduced) Bridgeland's Hall algebra of \emph{A}, which is the Ringel--Hall algebra of $2$-cyclic complexes over projective $A$-modules with some localization (and reduction). He proved that for any hereditary algebra \emph{A} the quantized enveloping algebra associated to \emph{A} can be embedded into the reduced Bridgeland's Hall algebra of $A$. This provides a beautiful realization of the full quantized enveloping algebra by Hall algebras. In \cite{Br}, Bridgeland also showed that for the hereditary algebra \emph{A} there exists an embedding of algebras from the Ringel--Hall algebra to Bridgeland's Hall algebra. Later on, Geng and Peng \cite{GP} generalized this embedding to algebras of global dimension at most two. So far, the structure of Bridgeland's Hall algebra of a nonhereditary algebra has been rarely known. Particularly, the relations between the Ringel--Hall algebra and Bridgeland's Hall algebra of any algebra with finite global dimension have not been uncovered.

In this paper, let $k$ be a finite field and $\A$ be an abelian $k$-category with enough projectives, which is of finite global dimension. We introduce certain elements in Bridgeland's Hall algebra of $\A$, and then give a sufficient and necessary condition for the multiplicative properties of these elements. As applications, we obtain \cite[Theorem 3.7]{GP}. Moreover, we generalize it to $2$th tilted algebras.

Throughout the paper, let $k$ be the finite field with $q$ elements and set $v=\sqrt{q}\in \mathbb{C}$. Let $\A$ be an abelian (small) $k$-category of finite global dimension, and assume that $\A$ has enough projectives. We denote by $\Iso(\A)$ the set of isomorphism classes of objects in $\A$, denote by $K(\A)$ the Grothendieck group of $\A$, and denote by $\mathscr{P}$ the full subcategory of $\A$ consisting of projective objects. For a complex $M_\bullet=\cdots\longrightarrow M_i\longrightarrow M_{i+1}\longrightarrow\cdots$ of $\A$, its homology is denoted by $H_\ast(M_\bullet)$. For an object $M$ in $\A$, the class of $M$ in $K(\A)$ is denoted by $\hat{M}$.
Let $A$ be a finite dimensional $k$-algebra of finite global dimension and denote by $\mod A$ the category of finite dimensional (left) $A$-modules. For a finite set $S$, we denote by $|S|$ its cardinality. For each rational number $x$, $[x]$ is defined to be the largest integer not greater than $x$. We denote the quotient ring $\mathbb{Z}/2\mathbb{Z}$ by $\mathbb{Z}_2=\{0,1\}$.

\section{Preliminaries}
In this section, we recall the definitions and some necessary results of $2$-cyclic complexes, Ringel--Hall algebras and Bridgeland's Hall algebras. All of the materials can be
found in \cite{Br}, \cite{ZHC0,ZHC1,ZHC2}.
\subsection{2-cyclic complexes}
Let $\mathcal{C}_2(\mathcal{A})$ be the abelian category of 2-cyclic complexes over $\mathcal{A}$. The objects of this category consist of diagrams
\begin{equation*}M_\bullet=\xymatrix{{M_1}\ar@<0.7ex>[r]^{d_1^M}&{M_0}\ar@<0.7ex>[l]^{d_0^M}}\end{equation*}
in $\mathcal{A}$ such that $d_{i+1}^M\circ d_{i}^M=0, ~i\in \mathbb{Z}_2$.
A morphism $s_\bullet: M_\bullet \rightarrow N_\bullet$ consists of a diagram
\begin{equation*}\xymatrix{{M_1}\ar@<0.7ex>[r]^{d_1^M}\ar[d]_{s_1}&{M_0}\ar@<0.7ex>[l]^{d_0^M}\ar[d]^{s_0}\\
{N_1}\ar@<0.7ex>[r]^{d_1^N}&{N_0}\ar@<0.7ex>[l]^{d_0^N}}\end{equation*}
with $s_{i+1}\circ d_i^M=d_i^N\circ s_i, ~i\in \mathbb{Z}_2$.
Two morphisms $s_\bullet, t_\bullet : M_\bullet \rightarrow N_\bullet$ are said to be \emph{homotopic} if there are morphisms $h_i : M_i \rightarrow N_{i+1},~i\in \mathbb{Z}_2,$ such that $t_i-s_i={d_{i+1}^N}\circ h_i+{h_{i+1}}\circ d_i^M,~i\in \mathbb{Z}_2$.
For an object $M_\bullet\in\mathcal{C}_2(\mathcal{A})$, we define its class in the Grothendieck group $K(\mathcal{A})$ to be
$$\hat{M_\bullet} := \hat{M_0}-\hat{M_1}\in K(\mathcal{A}).$$
Denote by ${\mathcal {K}}_2(\mathcal{A})$ the homotopy category obtained from $\mathcal{C}_2(\mathcal{A})$ by identifying homotopic morphisms. Let $\mathcal{C}_2(\mathscr{P})\subset\mathcal{C}_2(\A)$ be the full subcategory whose objects are complexes of projectives, and denote by ${\mathcal {K}}_2(\mathscr{P})$ its homotopy category.
The shift functor of complexes induces an involution of $\mathcal{C}_2(\mathcal{A})$. This involution
shifts the grading and changes the signs of differentials as follows
\[{M_ \bullet } = \xymatrix{{M_1}\ar@<0.7ex>[r]^{d_1^M}&{M_0}\ar@<0.7ex>[l]^{d_0^M}}\overset {*} \longleftrightarrow M_ \bullet ^* = \xymatrix{{M_0}\ar@<0.7ex>[r]^{-d_0^M}&{M_1}\ar@<0.7ex>[l]^{-d_1^M}.}\]

It is well-known that there exists an exact functor $\pi: \mathcal{C}^b(\mathcal{A})\longrightarrow\mathcal{C}_2(\mathcal{A})$, sending a complex $(M_i)_{i\in\mathbb{Z}}$ to the 2-cyclic complex
$$\xymatrix{{\bigoplus\limits_{i\in\mathbb{Z}}M_{2i+1}}\ar@<0.7ex>[r]&{\bigoplus\limits_{i\in\mathbb{Z}}M_{2i}}\ar@<0.7ex>[l]}$$
with the naturally defined differentials. It is easy to check that
$$\Hom_{\mathcal{C}_2(\mathcal{A})}(\pi(M_\bullet),\pi(N_\bullet))\cong\bigoplus\limits_{i\in\mathbb{Z}}\Hom_{\mathcal{C}^b(\mathcal{A})}(
M_\bullet,N_\bullet[2i]).$$

Let $\mathcal{D}^b(\A)$ be the bounded derived category of $\A$ with the suspension functor $[1]$. Let $\mathcal{R}_2(\A)$ = $\mathcal{D}^b(\A)/[2]$ be the orbit category, also known as the root category of $\mathcal{A}$. The category $\mathcal{D}^b(\A)$ is equivalent to the bounded homotopy category $K^b(\mathscr{P})$, since $\A$ is of finite global dimension. In this case, we can equally well define  $\mathcal{R}_2(\A)$ as the orbit category of $K^b(\mathscr{P})$.

\begin{lemma}{\rm(\cite{PX97}, \cite[Lemma 3.1]{Br})}\label{fully faithful}
There is a fully faithful functor $\mathcal {F} : \mathcal{R}_2(\A)\rightarrow {\mathcal {K}}_2(\mathscr{P})$ sending a bounded complex of projectives $(P_i)_{i \in \mathbb{Z}}$ to the $2$-cyclic complex
$$\xymatrix{{\bigoplus_{i\in\mathbb{Z}} P_{2i+1}}\ar@<0.7ex>[r]^{}&{\bigoplus_{i\in\mathbb{Z}} P_{2i}}\ar@<0.7ex>[l]^{}.}$$
\end{lemma}

The following lemma converts the calculation of dimensions of $\Ext$-spaces into that of $\Hom$-spaces.
\begin{lemma}{\rm(\cite[Lemma 3.3]{Br})}\label{Ext to Hom}
If $M_\bullet,N_\bullet \in \mathcal{C}_2(\mathscr{P})$, then there exists an isomorphism of vector spaces $$\Ext_{\mathcal{C}_2(\A)}^1(N_\bullet,M_\bullet) \cong \Hom_{{\mathcal {K}}_2(\A)}(N_\bullet,M_\bullet^\ast).$$
\end{lemma}

A complex $M_\bullet \in \mathcal{C}_2(\A)$ is called \emph{acyclic} if $H_\ast(M_\bullet)=0$. Each object $P \in \mathscr{P}$ determines acyclic complexes
\[{K_P} = (\xymatrix{{P}\ar@<0.7ex>[r]^{1}&{P}\ar@<0.7ex>[l]^{0}}),  ~~~~~~~K_P^* = (\xymatrix{{P}\ar@<0.7ex>[r]^{0}&{P}\ar@<0.7ex>[l]^{1}}).\]

\begin{lemma}{\rm(\cite[Lemma 3.2]{Br})}\label{zero}
For each acyclic complex $M_\bullet \in \mathcal{C}_2(\mathscr{P})$, there are objects $P,Q \in \mathscr{P}$, unique up to isomorphism, such that $M_\bullet \cong K_P \bigoplus K_Q^*$.
\end{lemma}

\subsection{Ringel--Hall algebras}
 Given objects $L,M,N \in \mathcal{A}$, let $\Ext_\mathcal{A}^1(M,N)_L \subset \Ext_\mathcal{A}^1(M,N)$ be the subset consisting of those equivalence classes of short exact sequences with middle term $L$.
 From now on, we always assume that $\A$ is finitary, i.e., all $\Hom$-spaces and $\Ext$-spaces are finite dimensional.
\begin{definition}\label{Hall algebra of abelian category}
The \emph{Hall algebra} $\mathcal {H}(\mathcal{A})$ of $\mathcal{A}$ is the vector space over $\mathbb{C}$ with basis elements $[M] \in \Iso(\mathcal{A}$), and with multiplication defined by
\[[M] \diamond [N] = \sum\limits_{[L] \in \Iso(\mathcal{A})} {\frac{{|\Ext_\mathcal{A}^1{{(M,N)}_L}|}}{{|\Hom_\mathcal{A}(M,N)|}}} [L].\]
\end{definition}

For objects $M,N \in \mathcal{A}$, the {\it Euler form} $\lr{-,-}:K(\A)\times K(\A)\ra
\bbz$ associated with $\mathcal{A}$ is defined by
$$\lr{\hat{M},\hat{N}}:=\sum_{i\in\mathbb{Z}_{\geq0}}(-1)^i\dim_k\Ext^i_{\A}(M,N).$$
The \emph{symmetric Euler form} $(\cdot ,\cdot ): K(\mathcal{A})\times K(\mathcal{A})\longrightarrow \mathbb{Z},$ is given by $$(\alpha,\beta)=\lr{\alpha,\beta}+\lr{\beta,\alpha}$$
for all $\alpha,\beta \in K(\mathcal{A})$.
The \emph{Ringel--Hall algebra} ${\mathcal {H}}_{\rm{tw}}(\mathcal{A})$ of $\mathcal{A}$ is the same vector space as $\mathcal {H}(\mathcal{A})$, but with multiplication defined by $$[M]\ast[N]=v^{\lr{\hat{M},\hat{N}}}\cdot[M]\diamond[N].$$

\subsection{Bridgeland's Hall algebras}
Let $\mathcal {H}(\mathcal{C}_2(\mathcal{A}))$ be the Hall algebra of the abelian category $\mathcal{C}_2(\mathcal{A})$ defined in Definition \ref{Hall algebra of abelian category} and $\mathcal {H}(\mathcal{C}_2(\mathscr{P})) \subset \mathcal {H}(\mathcal{C}_2(\mathcal{A}))$ be the subspace spanned by the isoclasses of complexes of projectives. Define $\mathcal {H}_{\rm{tw}}(\mathcal{C}_2(\mathscr{P}))$ to be the same vector space as $\mathcal {H}(\mathcal{C}_2(\mathscr{P}))$ with ``twisted" multiplication defined by $$[M_\bullet]\ast[N_\bullet]:=v^{\lr{{\hat{M}}_0,{\hat{N}}_0}+\lr{{\hat{M}}_1,{\hat{N}}_1}}\cdot[M_\bullet]\diamond[N_\bullet].$$ Then $\mathcal {H}_{\rm{tw}}(\mathcal{C}_2(\mathscr{P}))$ is an associative algebra (cf. \cite{Br}).

We have the following simple relations for the acyclic complexes ${K_P}$ and ${K_P^*}$.

\begin{lemma}{\rm(\cite[Lemma 3.5]{Br})}\label{formula}
For any object $P \in \mathscr{P}$ and any complex ${M_\bullet}\in\mathcal{C}_2(\mathscr{P})$, we have the following relations in $\mathcal {H}_{\rm{tw}}(\mathcal{C}_2(\mathscr{P}))$
\begin{alignat}{2}
&[K_P]\ast[M_\bullet]=v^{\lr{\hat{P},\hat{M}_\bullet}}[K_P \oplus M_\bullet],&\quad&[M_\bullet]\ast[K_P]=v^{-\lr{\hat{M}_\bullet,\hat{P}}}[K_P \oplus M_\bullet];\\
&[K_P^\ast]\ast[M_\bullet]=v^{-\lr{\hat{P},\hat{M}_\bullet}}[K_P^\ast \oplus M_\bullet],&\quad&[M_\bullet]\ast[K_P^\ast]=v^{\lr{\hat{M}_\bullet,\hat{P}}}[K_P^\ast \oplus M_\bullet];\\
&[K_P]\ast[M_\bullet]=v^{(\hat{P},\hat{M}_\bullet)}[M_\bullet]\ast[K_P],&\quad&[K_P^\ast]\ast[M_\bullet]=v^{-(\hat{P},\hat{M}_\bullet)}[M_\bullet]\ast[K_P^\ast]\label{community of K_P with others}.
\end{alignat}
In particular, for $P,Q \in \mathscr{P}$, we have
\begin{flalign}&[K_P]\ast[K_Q]=[K_P \oplus K_Q],~~~~[K_P]\ast[K_Q^\ast]=[K_P \oplus K_Q^\ast];\\
&[[K_P],[K_Q]]=[[K_P],[K_Q^\ast]]=[[K_P^\ast],[K_Q^\ast]]=0.\end{flalign}
\end{lemma}

By Lemma \ref{zero} and Lemma \ref{formula}, the acyclic elements of $\mathcal {H}_{\rm{tw}}(\mathcal{C}_2(\mathscr{P}))$ satisfy the Ore conditions and thus we  have the following definition; See \cite{Br}.
\begin{definition}
The \emph{Bridgeland's Hall algebra} of $\mathcal{A}$, denoted by $\mathcal {D}\mathcal {H}_2(\mathcal{A})$, is the localization of $\mathcal {H}_{\rm{tw}}(\mathcal{C}_2(\mathscr{P}))$ with respect to the elements $[M_\bullet]$ corresponding to acyclic complexes $M_\bullet$. In symbols, $$\mathcal {D}\mathcal {H}_2(\mathcal{A}):=\mathcal {H}_{\rm{tw}}(\mathcal{C}_2(\mathscr{P}))[~[M_\bullet]^{-1}~|~H_\ast(M_\bullet)=0~].$$
\end{definition}
As explained in \cite{Br}, this is the same as localizing by the elements $[K_P]$ and $[K_P^\ast]$ for all objects $P \in \mathscr{P}.$
Writing $\alpha \in K(\mathcal{A})$ in the form $\alpha = \hat{P}-\hat{Q}$ for some objects $P,Q \in \mathscr{P}$, one defines $K_\alpha = [K_P]\ast[K_Q]^{-1}, K_\alpha^\ast = [K_P^\ast]\ast[K_Q^\ast]^{-1}$. Note that the relations in (\ref{community of K_P with others}) are still satisfied with the elements $[K_P]$ and $[K_P^\ast]$ replaced by $K_\alpha$ and $K_\alpha^\ast$, respectively, for any $\alpha \in K(\mathcal{A})$.

\section{Main result}
\subsection{Projective resolutions}
For each object $M\in\A$, take a projective resolution of $M$
\begin{equation}\label{tsfj}\xymatrix{P'_M:0\ar[r]&P'_n\ar[r]^-{p'_n}&P'_{n-1}\ar[r]^-{p'_{n-1}}&\cdots\ar[r]^-{p'_{2}}&
P'_{1}\ar[r]^-{p'_{1}}&P'_{0}\ar[r]^-{p'_{0}}&M\ar[r]&0.}\end{equation}
Then we obtain the complex
$$\xymatrix{{P'_M}_\bullet:\cdots\ar[r]&0\ar[r]&P'_n\ar[r]^-{p'_n}&P'_{n-1}\ar[r]^-{p'_{n-1}}&\cdots\ar[r]^-{p'_{2}}&
P'_{1}\ar[r]^-{p'_{1}}&P'_{0}\ar[r]^-{p'_{0}}&0\ar[r]&\cdots,}$$
where $P'_{i},0\leq i\leq n$, are of degree $-i$.
It is similar to \cite[Section 4.1]{Br} that we can define the minimal projective resolutions. It is easy to see that each object has a unique minimal projective resolution up to isomorphism.

Let
\begin{equation}\label{jx}\xymatrix{{P}_M:0\ar[r]&P_n\ar[r]^-{p_n}&P_{n-1}\ar[r]^-{p_{n-1}}&\cdots\ar[r]^-{p_{2}}&
P_{1}\ar[r]^-{p_{1}}&P_{0}\ar[r]^-{p_{0}}&M\ar[r]&0}\end{equation} be the minimal projective resolution of $M$.
Consider the corresponding $2$-cyclic complex $$C_M:=\pi({P_M}_\bullet)\in\mathcal{C}_2(\mathscr{P}).$$
Set $P_M^{\rm even}:=\bigoplus\limits_{0\leq i=2m\leq n}P_i$ and $P_M^{\rm odd}:=\bigoplus\limits_{0\leq i=2m+1\leq n}P_i$. Then $C_M$ is the complex
$$\xymatrix{{P_M^{\rm{odd}}}\ar@<0.7ex>[r]&{P_M^{\rm{even}}}\ar@<0.7ex>[l]}$$ with the naturally defined differentials.

It is similar to \cite[Lemma 4.1]{Br} that we have the following
\begin{lemma}\label{form}
Any resolution (\ref{tsfj}) is isomorphic to a resolution of the form
\begin{multline}\label{xingshi}\xymatrix{0\ar[r]&P_n\oplus R_{n-1}\ar[rr]^-{{\begin{pmatrix}\begin{smallmatrix}p_n&0\\0&0\\0&1\end{smallmatrix}\end{pmatrix}}}&&P_{n-1}\oplus R_{n-2}\oplus R_{n-1}\ar[rr]^-{{\begin{pmatrix}\begin{smallmatrix}p_{n-1}&0&0\\0&0&0\\0&1&0\end{smallmatrix}\end{pmatrix}}}&&\cdots}\\
\xymatrix{\cdots
\ar[rr]^-{{\begin{pmatrix}\begin{smallmatrix}p_2&0&0\\0&0&0\\0&1&0\end{smallmatrix}\end{pmatrix}}}&&P_1\oplus R_0\oplus R_1\ar[rr]^-{\begin{pmatrix}\begin{smallmatrix}p_1&0&0\\0&1&0\end{smallmatrix}\end{pmatrix}}&&P_0\oplus R_0\ar[rr]^-{\left(\begin{smallmatrix}p_0&0&0\end{smallmatrix}\right)}&&M\ar[r]&0}\end{multline}
for some objects $R_i\in\mathscr{P},0\leq i<n$, and some minimal projective resolution
$$\xymatrix{0\ar[r]&P_n\ar[r]^-{p_n}&P_{n-1}\ar[r]^-{p_{n-1}}&\cdots\ar[r]^-{p_{2}}&
P_{1}\ar[r]^-{p_{1}}&P_{0}\ar[r]^-{p_{0}}&M\ar[r]&0.}$$
\end{lemma}

\subsection{Certain elements in $\mathcal {D}\mathcal {H}_2(\mathcal{A})$}
For each $M\in\A$
and any projective resolution $P'_M$ (\ref{tsfj}), set $M'_i:=\im p'_i, 0\leq i\leq n$. Let $M'_{\rm {even}}:=\bigoplus\limits_{1\leq i=2m\leq n}M'_i$ and $M'_{\rm{odd}}:=\bigoplus\limits_{1\leq i=2m+1\leq n}M'_i$. Then we define
$$E_{P'_M}:=v^{\lr{\hat{M}'_{\rm{odd}}-\hat{M}'_{\rm {even}},\hat{M}}}K_{-\hat{M}'_{\rm{odd}}}K_{-\hat{M}'_{\rm {even}}}^\ast [C'_M],$$
where $C'_M=\pi({P'_M}_\bullet)$. Similarly, for the minimal projective resolution $P_M$ (\ref{jx}), we can define $M_i,0\leq i\leq n$, and thus $M_{\rm {even}},M_{\rm {odd}}$.

\begin{lemma}\label{ldy}
For each $M\in\A$ and any two projective resolutions $P'_M$ and $P''_M$,
$E_{P'_M}=E_{P''_M}$.
\end{lemma}
\begin{proof}
By Lemma \ref{form}, each projective resolution of $M$ is of the form (\ref{xingshi}) up to isomorphism.
So it suffices to prove that for the projective resolution $P'_M$ (\ref{xingshi}) and the corresponding minimal projective resolution $P_M$ (\ref{jx}), $E_{P_M}=E_{P'_M}$.

It is easy to see that for each $1\leq i\leq n$,
$$M'_i\cong M_i\oplus R_{i-1}~~\text{and}~~C'_M\cong K_{R_{\rm even}}\oplus K_{R_{\rm odd}}^\ast\oplus C_M,$$ where
$R_{\rm {even}}:=\bigoplus\limits_{0\leq i=2m< n}R_i$ and $R_{\rm{odd}}:=\bigoplus\limits_{1\leq i=2m+1< n}R_i$. Thus, $M'_{\rm {even}}\cong M_{\rm {even}}\oplus R_{\rm{odd}}$ and $M'_{\rm {odd}}\cong M_{\rm {odd}}\oplus R_{\rm{even}}$.
By Lemma \ref{formula},
$$[C'_M]=[K_{R_{\rm even}}\oplus K_{R_{\rm odd}}^\ast\oplus C_M]=
v^{\lr{\hat{R}_{\rm odd}-\hat{R}_{\rm even},\hat{M}}}K_{\hat{R}_{\rm even}}K_{\hat{R}_{\rm odd}}^\ast[C_M].$$
Hence, \begin{equation*}
\begin{split}
E_{P'_M}&=v^{\lr{\hat{M}'_{\rm{odd}}-\hat{M}'_{\rm {even}},\hat{M}}}K_{-\hat{M}'_{\rm{odd}}}K_{-\hat{M}'_{\rm {even}}}^\ast [C'_M]\\
&=v^{\lr{\hat{M}'_{\rm{odd}}-\hat{M}'_{\rm {even}},\hat{M}}+\lr{\hat{R}_{\rm odd}-\hat{R}_{\rm even},\hat{M}}}K_{-\hat{M}'_{\rm{odd}}}K_{-\hat{M}'_{\rm {even}}}^\ast K_{\hat{R}_{\rm even}}K_{\hat{R}_{\rm odd}}^\ast[C_M]\\
&=v^{\lr{\hat{M}_{\rm{odd}}-\hat{M}_{\rm {even}},\hat{M}}}K_{-\hat{M}_{\rm{odd}}}K_{-\hat{M}_{\rm {even}}}^\ast [C_M]\\
&=E_{P_M}.
\end{split}
\end{equation*}
\end{proof}
By Lemma \ref{ldy}, for any  projective resolution $P_M$ we can define
$$E_{M}:=v^{\lr{\tau(M),\hat{M}}}K_{-\hat{M}_{\rm{odd}}}K_{-\hat{M}_{\rm {even}}}^\ast [C_M],$$
where $\tau(M):=\hat{M}_{\rm{odd}}-\hat{M}_{\rm {even}}$.

\subsection{Multiplicative properties of elements $E_M$}
\begin{lemma}\label{yin1}
$|\Hom_{\C^b(\A)}({P_M}_\bullet,{P_N}_\bullet)|=\prod\limits_{i=0}^{n-1}|\Hom(P_i,N_{i+1})|\cdot|\Hom_{\A}(M,N)|.$
\end{lemma}
\begin{proof}
For $f=(f_0,\cdots,f_n)\in\Hom_{\C^b(\A)}({P_M}_\bullet,{P_N}_\bullet)$, consider the following commutative diagram
\begin{equation}\label{tu}\xymatrix@C=.9em{0\ar[r]&P_n\ar[d]_{f_n}\ar[r]^{p_n}
&P_{n-1}\ar[d]^{f_{n-1}}\ar[r]^{p_{n-1}}&\cdots\ar[r]^{p_3}&P_2\ar[d]_{f_2}\ar[rr]^{p_2}
&&P_1\ar[d]^{f_1}\ar[rr]^{p_1}\ar@{-->}@/_/[ddl]&&P_0\ar[d]^{f_0}\ar[r]^{p_0}\ar@{-->}@/_/[ddl]
&M\ar[d]\ar[r]&0\\0\ar[r]&Q_n\ar[r]^{q_n}&Q_{n-1}\ar[r]^{q_{n-1}}&\cdots\ar[r]^{q_3}&Q_2\ar[rr]^{q_2}_{g_1}\ar@{->>}[dr]_{\pi_2}
&&Q_1\ar[rr]^{q_1}_{g_0}\ar@{->>}[dr]_{\pi_1}&&Q_0\ar[r]^{q_0}&N\ar[r]&0\\&&&&&N_2\ar@{^{(}->}[ru]_{i_1}&&N_1\ar@{^{(}->}[ru]_{i_0}&&&.}
\end{equation}
By the Comparison Lemma (See \cite[Theorem 2.2.6]{Weibel}), we have a surjective map
$$\xymatrix{\varphi_1: \Hom_{\C^b(\A)}({P_M}_\bullet,{P_N}_\bullet)\ar@{->>}[r]& \Hom_{\A}(M,N).}$$
Then $K_1:=\Ker\varphi_1=\{(f_0,\cdots,f_n)\in\Hom_{\C^b(\A)}({P_M}_\bullet,{P_N}_\bullet)~|~q_0f_0=0\}$.
Set
$$S_0:=\{(g_0,f_1,\cdots,f_n)~|~\pi_1f_1=g_0p_1,f_ip_{i+1}=q_{i+1}f_{i+1},i=1,\cdots,n-1\}.$$
It is easy to obtain the following injective map
$$\xymatrix{\varphi_2: K_1\ar@{^{(}->}[r]&S_0,&(f_0,f_1,\cdots,f_n)\ar@{|->}[r]&(g_0,f_1,\cdots,f_n),}$$
where $g_0$ is uniquely determined as in the diagram (\ref{tu}) by the universal property of the kernel of $q_0$, since $q_0f_0=0$.

Conversely, for any $(g_0,f_1,\cdots,f_n)\in S_0$, set $f_0=i_0g_0$, then
$f_0p_1=i_0g_0p_1$, $q_1f_1=i_0\pi_1f_1$. Since $g_0p_1=\pi_1f_1$, we obtain that $f_0p_1=q_1f_1$.
Hence, $\varphi_2$ is bijective, and thus
$$|\Hom_{\C^b(\A)}({P_M}_\bullet,{P_N}_\bullet)|=|S_0|\cdot|\Hom_{\A}(M,N)|.$$
Set $S_1:=\{(g_1,f_2,\cdots,f_n)~|~\pi_2f_2=g_1p_2,f_ip_{i+1}=q_{i+1}f_{i+1},i=2,\cdots,n-1\}$. We have the following injective map
$$\xymatrix{\varphi_3: S_1\ar@{^{(}->}[r]&S_0,&(g_1,f_2,\cdots,f_n)\ar@{|->}[r]&(0,f_1=i_1g_1,f_2,\cdots,f_n).}$$
It is easy to check directly that there is a short exact sequence
$$\xymatrix{0\ar[r]&S_1\ar[r]&S_0\ar[r]&\Hom_{\A}(P_0,N_1)\ar[r]&0.}$$
So, $$|S_0|=|S_1|\cdot|\Hom_{\A}(P_0,N_1)|.$$
Repeating the above process, we define $S_j:=\{(g_j,f_{j+1},\cdots,f_n)~|~\pi_{j+1}f_{j+1}=g_jp_{j+1},f_ip_{i+1}=q_{i+1}f_{i+1},i=j+1,\cdots,n-1\}, 2\leq j\leq n-2$, $S_{n-1}:=\{(g_{n-1},f_n)~|~\pi_{n}f_{n}=g_{n-1}p_{n}\}$, and $S_n=0$. Moreover, $|S_{j-1}|=|S_j|\cdot|\Hom_{\A}(P_{j-1},N_j)|, 2\leq j\leq n$. Therefore, by recursion, we complete the proof.\end{proof}

\begin{lemma}\label{Hom}
\begin{equation}\begin{split}&|\Hom_{\C_2(\A)}(C_M,C_N)|=\\&\prod\limits_{i=0}^{[\frac{n}{2}]}\prod\limits_{t=2i}^{n}(|\Hom_{\A}(M_{2i},N)|\cdot|\Hom_{\A}(P_t,N_{t-2i+1})|)
\cdot\prod\limits_{i=1}^{[\frac{n}{2}]}\prod\limits_{t=0}^{n-2i-1}|\Hom_{\A}(P_t,N_{t+2i+1})|.\end{split}\end{equation}
\end{lemma}
\begin{proof}
$$\Hom_{\C_2(\A)}(C_M,C_N)\cong\bigoplus_{i\in\mathbb{Z}}\Hom_{\C^b(\A)}({P_M}_\bullet,{P_N}_\bullet[2i]).$$
Clearly, if $2i>n$, $\Hom_{\C^b(\A)}({P_M}_\bullet,{P_N}_\bullet[2i])=\Hom_{\C^b(\A)}({P_M}_\bullet,{P_N}_\bullet[-2i])=0$.

For any $0\leq i\leq [\frac{n}{2}]$, \begin{equation*}\begin{split}|\Hom_{\C^b(\A)}({P_M}_\bullet,{P_N}_\bullet[2i])|&=|\Hom_{\C^b(\A)}({P_{M_{2i}}}_\bullet,{P_N}_\bullet)|\\&=
\prod\limits_{t=2i}^{n}(|\Hom_{\A}(M_{2i},N)|\cdot|\Hom_{\A}(P_t,N_{t-2i+1})|).\end{split}\end{equation*} In a similar way as Lemma \ref{yin1}, for any $1\leq i\leq [\frac{n}{2}]$, we obtain $$|\Hom_{\C^b(\A)}({P_M}_\bullet,{P_N}_\bullet[-2i])|=\prod\limits_{t=0}^{n-2i-1}|\Hom_{\A}(P_t,N_{t+2i+1})|.$$
Therefore, the proof is completed.
\end{proof}

\begin{lemma}\label{HE}
Let $M,N\in\A$. Then
\begin{equation}\begin{split}\frac{|\Hom_{\C_2(\A)}(C_M,C_N)|}{|\Hom_{\A}(M,N)|}&=q^{\lr{\hat{M}_{\rm {even}},\hat{N}}+\lr{\hat{P}_M^{\rm even},\hat{N}_{\rm {odd}}}+\lr{\hat{P}_M^{\rm odd},\hat{N}_{\rm even}}+w_0}.
\end{split}\end{equation}
where $w_0=\sum\limits_{i=1}^{[\frac{n}{2}]}\sum\limits_{t=2i+1}^{n}(-1)^{t-1}\dim\Ext^{t}_{\A}(M,N)$. In particular, if $w_0=0$,
then
$$\frac{|\Hom_{\C_2(\A)}(C_M,C_N)|}{|\Hom_{\A}(M,N)|}=v^{\lr{\hat{M},\tau(N)}-\lr{\tau(M),\hat{N}}+\lr{\hat{P}_M^{\rm odd},\hat{P}_N^{\rm odd}}+\lr{\hat{P}_M^{\rm even},\hat{P}_N^{\rm even}}-\lr{\hat{M},\hat{N}}}.$$
\end{lemma}
\begin{proof} Let
\begin{equation*}\begin{split}
r_1&=\dim\Hom_{\A}(M,N),\\
r_2&=\dim\Hom_{\C_2(\A)}(C_M,C_N),\\
r_3&=\sum\limits_{i=0}^{[\frac{n}{2}]}\dim\Hom_{\A}(M_{2i},N).
\end{split}\end{equation*}
By Lemma \ref{Hom},
\begin{equation*}\begin{split}
r_2=&r_3+\sum\limits_{i=0}^{[\frac{n}{2}]}\sum\limits_{t=2i}^{n}
\dim\Hom_{\A}(P_t,N_{t-2i+1})+\sum\limits_{i=1}^{[\frac{n}{2}]}\sum\limits_{t=0}^{n-2i+1}\dim\Hom_{\A}(P_t,N_{t+2i+1})\\
=&r_3+\sum_{t=0}^n\lr{\hat{P}_t,\hat{N}_{t+1}}+\sum\limits_{i=1}^{[\frac{n}{2}]}
(\sum\limits_{t=2i}^{n}\lr{\hat{P}_t,\hat{N}_{t-2i+1}}+\sum\limits_{t=0}^{n-2i+1}\lr{\hat{P}_t,\hat{N}_{t+2i+1}})\\
=&r_3+\lr{\hat{P}_0,\hat{N}_{\rm odd}}+\lr{\hat{P}_1,\hat{N}_{\rm even}}+\sum_{t=2}^n\lr{\hat{P}_t,\frac{1+(-1)^t}{2}\hat{N}_{\rm odd}+\frac{1-(-1)^t}{2}\hat{N}_{\rm even}}\\
=&r_3+\lr{\hat{P}_M^{\rm even},\hat{N}_{\rm {odd}}}+\lr{\hat{P}_M^{\rm odd},\hat{N}_{\rm even}}.
\end{split}\end{equation*}
In fact,
\begin{equation*}\begin{split}
r_3
&=\sum\limits_{i=1}^{[\frac{n}{2}]}\lr{\hat{M}_{2i},\hat{N}}+\sum\limits_{i=1}^{[\frac{n}{2}]}\sum\limits_{t=1}^{n-2i-1}(-1)^{t-1}\dim\Ext^{t}_{\A}(M_{2i},N)\\
&=\lr{\hat{M}_{\rm even},\hat{N}}+\sum\limits_{i=1}^{[\frac{n}{2}]}\sum\limits_{t=2i+1}^{n}(-1)^{t-1}\dim\Ext^{t}_{\A}(M,N)\\
&=\lr{\hat{M}_{\rm even},\hat{N}}+w_0.
\end{split}\end{equation*}
Therefore, $r_2-r_1=\lr{\hat{M}_{\rm {even}},\hat{N}}+\lr{\hat{P}_M^{\rm even},\hat{N}_{\rm {odd}}}+\lr{\hat{P}_M^{\rm odd},\hat{N}_{\rm even}}+w_0.$
So we have proved the first equality.

On the other hand, by definition, it is easy to see that
\begin{equation}\hat{P}_M^{\rm odd}=\hat{M}_{\rm odd}+\hat{M}_{\rm even}~~\text{and}~~\hat{P}_M^{\rm even}=\hat{M}_{\rm odd}+\hat{M}_{\rm even}+\hat{M}.\end{equation}
Hence, \begin{equation*}\begin{split}
&\lr{\hat{M},\tau(N)}-\lr{\tau(M),\hat{N}}+\lr{\hat{P}_M^{\rm odd},\hat{P}_N^{\rm odd}}+\lr{\hat{P}_M^{\rm even},\hat{P}_N^{\rm even}}-\lr{\hat{M},\hat{N}}\\&=
\lr{\hat{M},\hat{N}_{\rm odd}-\hat{N}_{\rm even}}-\lr{\hat{M}_{\rm odd}-\hat{M}_{\rm even},\hat{N}}+
\lr{\hat{M}_{\rm odd}+\hat{M}_{\rm even},\hat{N}_{\rm odd}+\hat{N}_{\rm even}}\\
&+\lr{\hat{M}_{\rm odd}+\hat{M}_{\rm even}+\hat{M},\hat{N}_{\rm odd}+\hat{N}_{\rm even}+\hat{N}}-
\lr{\hat{M},\hat{N}}\\
&=2\lr{\hat{M}_{\rm even},\hat{N}}+2\lr{\hat{M},\hat{N}_{\rm odd}}+2\lr{\hat{M}_{\rm odd}+\hat{M}_{\rm even},\hat{N}_{\rm odd}+\hat{N}_{\rm even}}\\
&=2\lr{\hat{M}_{\rm even},\hat{N}}+2\lr{\hat{P}_M^{\rm even}-\hat{P}_M^{\rm odd},\hat{N}_{\rm odd}}+2\lr{\hat{P}_M^{\rm odd},\hat{N}_{\rm odd}+\hat{N}_{\rm even}}\\
&=2(\lr{\hat{M}_{\rm {even}},\hat{N}}+\lr{\hat{P}_M^{\rm even},\hat{N}_{\rm {odd}}}+\lr{\hat{P}_M^{\rm odd},\hat{N}_{\rm even}}).
\end{split}\end{equation*}
Therefore we complete the proof.
\end{proof}

Let $\varphi$ be the $\mathbb{C}$-linear map defined by
$\varphi: {\mathcal {H}}_{\rm{tw}}(\mathcal{A})\longrightarrow \mathcal {D}\mathcal {H}_2(\mathcal{A}), [M]\longmapsto E_M$.
\begin{proposition}\label{zm}
The map $\varphi: {\mathcal {H}}_{\rm{tw}}(\mathcal{A})\longrightarrow \mathcal {D}\mathcal {H}_2(\mathcal{A})$ is injective. That is, $\{E_M~|~M\in\Iso(\A)\}$ is a set of linearly independent elements in $\mathcal {D}\mathcal {H}_2(\mathcal{A})$.
\end{proposition}
\begin{proof}
Let $\psi$ be the $\mathbb{C}$-linear map defined by
$$\psi: \mathcal {D}\mathcal {H}_2(\mathcal{A})\longrightarrow{\mathcal {H}}_{\rm{tw}}(\mathcal{A}),
[K_P\oplus K_Q^*]^{-1}\ast[M_\bullet]\longmapsto v^{\lr{\hat{\Ker} d_1-\hat{\im} d_1,\hat{H_0}(M_\bullet)}}[H_0(M_\bullet)],$$
where $P,Q\in\mathscr{P}$ and $M_\bullet=\xymatrix{{M_1}\ar@<0.7ex>[r]^{d_1}&{M_0.}\ar@<0.7ex>[l]^{d_0}}$
Then it is easy to check that $\psi\varphi=1$. Therefore, $\varphi$ is injective.
\end{proof}
\begin{theorem}\label{zd}
Let $M,N\in\A$. Then
$\varphi([M]\ast[N])=\varphi([M])\ast\varphi([N])$ if and only if 
$\Ext^i_{\A}(M,N)=0$ for any $i\geq 3$.
\end{theorem}
\begin{proof}
Note that
\begin{equation*}\begin{split}
\Ext^1_{\C_2(\A)}(C_M,C_N)&\cong\Hom_{{\mathcal {K}}_2(\mathscr{P})}(C_M,C_N[1])\\
&\cong\Hom_{\mathcal{R}_2(\A)}(M,N[1])\\&\cong\bigoplus\limits_{i\geq0}\Hom_{\mathcal{D}^b(\A)}(M,N[2i+1])\\&\cong
\bigoplus\limits_{i\geq0}\Ext^{2i+1}_{\A}(M,N).
\end{split}
\end{equation*}
Suppose $\Ext^i_{\A}(M,N)=0$ for any $i\geq 3$. Then $\Ext^1_{\C_2(\A)}(C_M,C_N)\cong \Ext_{\A}^1(M,N)$.
Let $0\longrightarrow N\longrightarrow L\longrightarrow M\longrightarrow 0$ be a short exact sequence, and $P_M, P_N$ be the minimal projective resolutions of $M$ and $N$, respectively. Then by Horse-shoe Lemma (see \cite[Section 2.2]{Weibel}) $P_M$ and $P_N$ induce a projective resolution $P'_L$ of $L$. Hence, we have the short exact sequence
$0\longrightarrow C_N\longrightarrow C'_L\longrightarrow C_M\longrightarrow 0$.
It is easy to see that
\begin{equation*}\hat{M}_{\rm{odd}}+\hat{N}_{\rm{odd}}=\hat{L}'_{\rm{odd}}~~\text{and}~~\hat{M}_{\rm {even}}+\hat{N}_{\rm {even}}=\hat{L}'_{\rm{even}}.\end{equation*}
Hence,
\begin{equation*}\begin{split}
&\varphi([M])\ast\varphi([N])\\&=v^{\lr{\tau(M),\hat{M}}+\lr{\tau(N),\hat{N}}}\ast K_{-\hat{M}_{\rm{odd}}}\ast K_{-\hat{M}_{\rm {even}}}^\ast \ast[C_M]\ast K_{-\hat{N}_{\rm{odd}}}\ast K_{-\hat{N}_{\rm {even}}}^\ast \ast[C_N]\\
&=v^{\lr{\tau(M),\hat{M}}+\lr{\tau(N),\hat{N}}+(\hat{N}_{\rm {odd}},\hat{M})-(\hat{N}_{\rm {even}},\hat{M})}\ast K_{-(\hat{M}_{\rm{odd}}+\hat{N}_{\rm{odd}})}\ast K_{-(\hat{M}_{\rm {even}}+\hat{N}_{\rm {even}})}^\ast\ast[C_M]\ast[C_N]\\
&=v^{t_0}\ast K_{-(\hat{M}_{\rm{odd}}+\hat{N}_{\rm{odd}})}\ast K_{-(\hat{M}_{\rm {even}}+\hat{N}_{\rm {even}})}^\ast\ast\sum_{[L]}\frac{|\Ext^1_{\C_2(\A)}(C_M,C_N)_{C'_L}|}{|\Hom_{\C_2(\A)}(C_M,C_N)|}\ast[C'_L]\\
&=\frac{v^{t_0}}{|\Hom_{\C_2(\A)}(C_M,C_N)|}\ast
\sum_{[L]}|\Ext^1_{\A}(M,N)_L|\ast K_{-\hat{L}'_{\rm{odd}}}\ast K_{-\hat{L}'_{\rm{even}}}^\ast \ast[C'_L]\\
&=\frac{v^{t_1}}{|\Hom_{\C_2(\A)}(C_M,C_N)|}\ast
\sum_{[L]}|\Ext^1_{\A}(M,N)_L|\ast E_L.
\end{split}
\end{equation*}
where $$t_0=\lr{\tau(M),\hat{M}}+\lr{\tau(N),\hat{N}}+(\tau(N),\hat{M})+\lr{\hat{P}_M^{\rm odd},\hat{P}_N^{\rm odd}}+\lr{\hat{P}_M^{\rm even},\hat{P}_N^{\rm even}}$$ and \begin{equation*}\begin{split}t_1&=t_0-\lr{\tau(M)+\tau(N),\hat{M}+\hat{N}}\\
&=\lr{\hat{M},\tau(N)}-\lr{\tau(M),\hat{N}}+\lr{\hat{P}_M^{\rm odd},\hat{P}_N^{\rm odd}}+\lr{\hat{P}_M^{\rm even},\hat{P}_N^{\rm even}}.\end{split}\end{equation*}

On the other hand,
\begin{equation}\label{Hall-mul}\begin{split}
\varphi([M]\ast[N])=\frac{v^{\lr{\hat{M},\hat{N}}}}{|\Hom_{\A}(M,N)|}\sum_{[L]}|\Ext^1_{\A}(M,N)_L|\ast E_L.
\end{split}\end{equation}
Note that $w_0=0$. By Lemma \ref{HE}, the sufficiency is proved.

Conversely, suppose $\varphi([M]\ast[N])=\varphi([M])\ast\varphi([N])$. In general,
\begin{equation}\label{zh}
\begin{aligned}
&\varphi([M])\ast\varphi([N])\\=&v^{t_0}\ast K_{-(\hat{M}_{\rm{odd}}+\hat{N}_{\rm{odd}})}\ast K_{-(\hat{M}_{\rm {even}}+\hat{N}_{\rm {even}})}^\ast\ast(\sum_{[L]}\frac{|\Ext^1_{\C_2(\A)}(C_M,C_N)_{C'_L}|}{|\Hom_{\C_2(\A)}(C_M,C_N)|}\ast[C'_L]\\
&+\sum_{[X_\bullet]\not=[C'_L]}\frac{|\Ext^1_{\C_2(\A)}(C_M,C_N)_{X_\bullet}|}{|\Hom_{\C_2(\A)}(C_M,C_N)|}\ast[X_\bullet])\\
=&\sum_{[L]}a_L(q) E_L+\sum_{[X^{\cdot}]\not=[C'_L]}b_{X^{\cdot}}(q)K_{-(\hat{M}_{\rm{odd}}+\hat{N}_{\rm{odd}})}\ast K_{-(\hat{M}_{\rm {even}}+\hat{N}_{\rm {even}})}^\ast\ast[X_\bullet],
\end{aligned}
\end{equation}
where $a_L(q)$ and $b_{X^{\cdot}}(q)$ belong to the rational function field $Q(q)$.
Since all $[C'_L]$ and $[X_\bullet]$ appearing in the sum are linearly independent in
$\mathcal {H}_{\rm{tw}}(\mathcal{C}_2(\mathscr{P}))$, it follows that all
$E_L$
and $K_{-(\hat{M}_{\rm{odd}}+\hat{N}_{\rm{odd}})}\ast K_{-(\hat{M}_{\rm {even}}+\hat{N}_{\rm {even}})}^\ast\ast[X_\bullet]$
are also linearly independent in
$\mathcal {D}\mathcal {H}_2(\mathcal{A})$. Comparing the formula (\ref{zh}) with (\ref{Hall-mul}), we get
that $\Ext^{2i+1}_{\A}(M,N)=0$ for any $i\geq 1$. Furthermore, by Lemma \ref{HE}, we conclude that $\Ext^{i}_{\A}(M,N)=0$ for
any $i\geq 3$.
\end{proof}

\begin{remark}
Define
$F_{M}:=v^{\lr{\tau(M),\hat{M}}}K_{-\hat{M}_{\rm {even}}}K_{-\hat{M}_{\rm{odd}}}^\ast [C_M^\ast],$ and
let $\varphi^\ast$ be the $\mathbb{C}$-linear map defined by
$\varphi^\ast: {\mathcal {H}}_{\rm{tw}}(\mathcal{A})\longrightarrow \mathcal {D}\mathcal {H}_2(\mathcal{A}), [M]\longmapsto F_M$. Then $(1)$ $\varphi^\ast$ is injective.
$(2)$ For any $M,N\in\A$,
$\varphi^\ast([M]\ast[N])=\varphi^\ast([M])\ast\varphi^\ast([N])$ if and only if $\Ext^i_{\A}(M,N)=0$ for any $i\geq 3$.
\end{remark}

In what follows, we only give the statements of results on $\varphi$, although they also hold for $\varphi^*$.
The following corollary is the main result of \cite{GP}.

\begin{corollary}
If $\A$ is of global dimension at most two. Then $\varphi$ is an embedding of ${\mathcal {H}}_{\rm{tw}}(\mathcal{A})$ into $\mathcal {D}\mathcal {H}_2(\mathcal{A})$.
\end{corollary}

In the following we apply our main result to iterated tilted algebras. We recall that an algebra $A$ is called an \emph{iterated tilted algebra} if there exists a sequence $(A_{i-1},T_{i-1},A_i)$, $i=1,2,\cdots,m$, where each $A_i$ is a finite dimensional $k$-algebra and $T_{i-1}$ is a (classical) tilting $A_{i-1}$-module (cf. \cite[Section VI.2]{ASS}), such that $A_0$ is hereditary, $A_i=\End_{A_{i-1}}(T_{i-1})$, and $A_m=A$. For convenience, we call $A_i$ \emph{ith tilted algebra}. We denote ${\mathcal {H}}_{\rm{tw}}(\mod A)$ and $\mathcal {D}\mathcal {H}_2(\mod A)$ by ${\mathcal {H}}_{\rm{tw}}(A)$ and $\mathcal {D}\mathcal {H}_2(A)$, respectively. It is well-known that $A_1$ is of global dimension at most two, so $$\varphi: {\mathcal {H}}_{\rm{tw}}(A)\longrightarrow \mathcal {D}\mathcal {H}_2(A),~~[M]\longmapsto E_M$$ is an embedding of algebras.

By the Brenner-Butler theorem (cf. \cite{ASS}), we know that the functors $\Hom_{A_1}(T_1,-)$ and $-\otimes_{A_2}T_1$ induce mutually inverse equivalences between the following two full subcategories of $\mod{A_1}$ and $\mod{A_2}$, respectively,
\begin{equation}
\mathcal {T}(T_1)=\{M~|~\Ext_{A_1}^1(T_1,M)=0\}\quad \mbox{and}\quad\mathcal {Y}(T_1)=\{N~|~\Tor_1^{A_2}(N,T)=0 \}.
\end{equation}
While the functors $\Ext_{A_1}^1(T,-)$ and $\Tor_1^{A_2}(-,T)$ induce mutually inverse equivalences between the following two full subcategories of $\mod{A_1}$ and $\mod{A_2}$, respectively,
\begin{equation}\label{XF}
\mathcal {F}(T_1)=\{M~|~\Hom_{A_1}(T_1,M)=0\}\quad \mbox{and}\quad\mathcal {X}(T_1)=\{N~|~N \otimes_{A_2} T_1=0 \}.
\end{equation}

\begin{corollary}
Let $A_2$ be the 2th tilted algebra as above. Then there exist injective homomorphisms $$\varphi_1: {\mathcal {H}}_{\rm{tw}}(\mathcal {X}(T_1))\longrightarrow\mathcal {D}\mathcal {H}_2(A_2),[M]\longmapsto E_M$$ and $$\varphi_2: {\mathcal {H}}_{\rm{tw}}(\mathcal {Y}(T_1))\longrightarrow\mathcal {D}\mathcal {H}_2(A_2),[M]\longmapsto E_M.$$
\end{corollary}
\begin{proof}
Let $M,N\in\mathcal {X}(T_1)$. By the Brenner-Butler theorem, for any $i\geq0$,
$$\Ext_{A_2}^{i}(M,N)\cong\Ext_{A_1}^{i}(M',N')$$ for some $M',N'\in\mathcal {F}(T_1)$.
For any $i\geq3$, since $A_1$ is of global dimension at most two, it follows that $\Ext_{A_1}^{i}(M',N')=0$ and then $\Ext_{A_2}^{i}(M,N)=0$. Then by Proposition \ref{zm} and Theorem \ref{zd}, $\varphi_1$ is an embedding of algebras. Similarly, we can prove the second embedding.
\end{proof}

\begin{example} Let $Q$ be the quiver
 $$\xymatrix{
           1\ar[r]&2\ar[r]&3\cdots \ar[r]&n+1}$$
and let $A=kQ/I$, where $I$ is the ideal of $kQ$
generated by all paths of length $2$. Observe that for any $a,b\in\mathbb{N}$, $0<\mid i-j\mid\leq 2$ and $t\geq 3$,
$\Ext^t(aS_i,bS_j)=0$. Let $N=1-(S_i,S_{j})$,
by Theorem \ref{zd}, we have the Serre relations
$$\sum\limits_{t=0}^{N}(-1)^t{\scriptsize\left[\begin{array}{cc}
    N\\
   t\end{array}\right]_v}E_{S_i}^{(t)}E_{S_j}E_{S_i}^{(N-t)}=0,$$
since $\sum\limits_{t=0}^{N}(-1)^t{\scriptsize\left[\begin{array}{cc}
    N\\
   t\end{array}\right]_v}[S_i]^{(t)}[S_j][S_i]^{(N-t)}=0$ holds in ${\mathcal {H}}_{\rm{tw}}(A)$ (cf. \cite[Section 10.2]{Deng}).
\end{example}

\section*{Acknowledgments}

The authors are grateful to Bangming Deng for his stimulating discussions and valuable comments.

\end{document}